\documentclass[12pt,reqno]{amsart}

\usepackage{amsthm, mathrsfs,amssymb,amsmath}
\usepackage{enumerate}
\usepackage[hidelinks]{hyperref}
\usepackage{xcolor}
\hypersetup{
	colorlinks,
	linkcolor={red!50!black},
	citecolor={green!50!black},
	urlcolor={red!80!black}
}
\makeatletter
\@namedef{subjclassname@2010}{%
	\textup{2010} Mathematics Subject Classification}
\makeatother

\allowdisplaybreaks

\newtheorem{thm}{Theorem}[section]

\newtheorem{prop}[thm]{Proposition}
\newtheorem{cor}[thm]{Corollary}

\theoremstyle{definition}
\newtheorem{defn}[thm]{Definition} 

\theoremstyle{remark}
\newtheorem{rem}{Remark}

\newcommand{\spt}{\textnormal{supp}}

\title[Hardy-Littlewood inequality and $L^p$-$L^q$ Fourier multipliers]{Hardy-Littlewood inequality and $L^p$-$L^q$ Fourier multipliers on compact hypergroups}
\author{Vishvesh Kumar}
\address{Vishvesh Kumar \endgraf
	Department of Mathematics: Analysis, Logic and Discrete Mathematics
	\endgraf
	Ghent University, Belgium}
\endgraf
\email{vishveshmishra@gmail.com}
\author[Michael Ruzhansky]{Michael Ruzhansky}
\address{
	Michael Ruzhansky:
	\endgraf
	Department of Mathematics: Analysis, Logic and Discrete Mathematics
	\endgraf
	Ghent University, Belgium
	\endgraf
	and
	\endgraf
	School of Mathematics
	\endgraf
	Queen Mary University of London
	\endgraf
	United Kingdom
	\endgraf
	{\it E-mail address} {\rm michael.ruzhansky@ugent.be}
}

\begin{document}
	
	\begin{abstract} This paper deals with the inequalities devoted to the comparison between the norm of a function on a compact hypergroup and the norm of its Fourier coefficients. We prove the classical Paley inequality in the setting of compact hypergroups which further gives the Hardy-Littlewood and Hausdorff-Young-Paley (Pitt) inequalities in the noncommutative context. We establish H\"ormander's $L^p$-$L^q$ Fourier multiplier theorem on compact hypergroups for $1<p \leq 2 \leq q<\infty$ as an application of Hausdorff-Young-Paley inequality. We examine our results for the hypergroups constructed from the conjugacy classes of compact Lie groups and for a class of countable compact hypergroups.       
	\end{abstract}
	\maketitle
	\tableofcontents 
	
	\section{Introduction}  
	\noindent The inequalities which involve functions and their Fourier coefficients played a pivotal role in Fourier analysis as well as in its applications to several different areas. This paper  contributes to some of the classical inequalities of this nature, namely, Hardy-Littlewood inequality, Paley inequality and Hausdorff-Young-Paley inequality, and their applications to the theory  of Fourier multiplier in the non-commutative setting. The first inequality we consider is the  Hardy-Littlewood inequality  proved by Hardy and Littlewood for the torus $ \mathbb{T}$ (\cite{Hardy}).  They proved that for each $ 1 \leq p \leq 2$ there exist a constant $C_p>0$ such that 
	\begin{equation}
	\left( \sum_{n \in \mathbb{Z}} |\widehat{f}(n)|^p\, (1+|n|)^{p-2}   \right)^{\frac{1}{p}} \leq C_p \|f\|_{L^p(\mathbb{T})},\,\quad f \in L^p(\mathbb{T}).
	\end{equation} 
	Hewitt and Ross \cite{HR74} extended this inequality for compact abelian groups using the structure theory of groups. Recently, the second author with his coauthors explored the non-commmutative version of the  Hardy-Littlewood inequality in the setting of compact homogeneous spaces \cite{ARN1,ARN} and compact quantum groups \cite{AMR} (see also \cite{Youn}). The Hardy-Littlewood inequality also has an application to Sobolev embedding theorems and to the  boundedness of Fourier multipliers \cite{Youn,Benyi, ARN}. Compact Riemannian spaces can be viewed as homogeneous spaces of compact Lie groups. It is well-known that the   spherical analysis on  Riemannian symmetric spaces is interconnected with the analysis on  the double coset spaces which are special examples of hypergroups for which a convolution structure can be defined on the space of all bounded Borel measures. Our goal is to investigate Hardy-Littlewood, Paley and Hausdorff-Young-Paley inequalities and their applications to the boundedness of Fourier multipliers in the context of compact hypergroups. The results of this paper are not only applicable to compact double coset spaces but also to the large class of other examples, for instance, the space of group orbits, space of conjugacy classes of compact (Lie) groups and countable compact hypergroups \cite{Bloom}. In particular, the results of this paper are also true for several interesting examples including Jacobi hypergroups with Jacobi polynomials as characters \cite{Gasper}, compact hypergroup structure on the fundamental alcove with Heckman-Opdam polynomials as characters \cite{HRos} and multivariant disk hypergroups \cite{RV, AnT} 
	
	Hewitt and Ross \cite{HR74} used structure theory of compact abelian groups and in \cite{ARN}, the authors used the eigenvalue counting formula for Laplace operator on compact manifolds to derive Hardy-Littlewood inequality. When working with compact hypergroups, we do not have such luxury. In this case,  we obtain the following Hardy-Littlewood inequality. 
	\begin{thm} \label{HLGintro}
		Let $1 <p \leq 2$ and let $K$ be a compact hypergroup. Assume that a sequence $\{\mu_\pi \}_{\pi \in \widehat{K}}$ grows sufficiently fast, that is, 
		\begin{equation} 
		\sum_{\pi \in \widehat{K}} \frac{k_\pi^2}{|\mu_\pi|^\beta}<\infty\,\,\,\, \text{for some}\,\beta \geq 0.
		\end{equation}  Then we have 
		\begin{equation} 
		\sum_{\pi \in \widehat{K}} k_\pi^2 |\mu_\pi|^{\beta (p-2)} \left( \frac{\|\widehat{f}(\pi)\|_{\textnormal{HS}}}{\sqrt{k_\pi}} \right)^p \lesssim \|f\|_{L^p(K)}.
		\end{equation}
	\end{thm}
	
	In the case when $K$ is the hypergroup of conjugacy classes of compact Lie group $\text{SU(2)}$ then Theorem \ref{HLGintro} gives the following Hardy-Littlewood inequality the commutative hypergroup $\textnormal{Conj(SU)(2)},$ which is a natural analogue of Hardy-Littlewood inequality for  $\mathbb{T}.$  
	\begin{thm} \label{HLineqintro}
		If $1 < p \leq 2$ and $f \in L^p(\textnormal{Conj(SU)(2)}),$ then we have 
		\begin{equation} \label{in4}
		\sum_{l \in \frac{1}{2}\mathbb{N}_0} (2l+1)^{5p-8} |\widehat{f}(l)|^{p} \leq C_p \|f\|_{L^p(\textnormal{Conj(SU)(2)})}.
		\end{equation}
	\end{thm}
The inequality \eqref{in4} can be interpreted in the following form similar to the Hardy-Littlewood inequality on $\mathbb{T}$: 
\begin{equation} \label{in5}
\sum_{l \in \frac{1}{2}\mathbb{N}_0} (2l+1)^{5(p-2)}  (2l+1)^2 |\widehat{f}(l)|^{p} \leq C_p \|f\|_{L^p(\textnormal{Conj(SU)(2)})}.
\end{equation}
In contrast to the case of $\mathbb{T},$ an extra term $(2l+1)^2$ appears in the above Inequality \eqref{in5}. But this is natural as the Plancherel measure $\omega$ on $\frac{1}{2}\mathbb{N}_0,$ the dual of $\textnormal{Conj(SU)(2)},$ is given by $\omega(l)= (2l+1)^{2}$ for $l \in \frac{1}{2} \mathbb{N}_0$ while for $\mathbb{T},$ the Plancherel  measure of the dual group $\mathbb{Z}$ is the counting measure. 
	\begin{cor} \label{corhl}
		If  $2 \leq p <\infty$ and $\sum_{l \in \frac{1}{2}\mathbb{N}_0} (2l+1)^{5p-8} |\widehat{f}(l)|^{p}<\infty$ then $$f \in L^p(\textnormal{Conj(SU)(2)}).$$ Moreover, we have 
		$$\|f\|_{L^p(\textnormal{Conj(SU)(2)})} \leq C_p \sum_{l \in \frac{1}{2}\mathbb{N}_0} (2l+1)^{5p-8} |\widehat{f}(l)|^{p}.$$
	\end{cor}

For $p=2,$ Theorem \ref{HLineqintro} and Corollary \ref{corhl} boil down to the  Plancherel theorem for the hypergroup $\textnormal{Conj(SU)(2)}.$ Therefore, these follow the philosophy of Hardy and Littlewood \cite{Hardy} who argue that Hardy-Littlewood inequality is a suitable extension of the Plancherel theorem in the case of $\mathbb{T}.$ 
 
 Another set of interesting examples of the commutative infinite hypergroups which we will investigate is the family of countable compact hypergroups studied by Dunkl and Ramirez \cite{dun2}. Recently, in \cite{KKA,KKAadd} first author with Singh and Ross studied classification results of such classes of hypergroups arising from the discrete semigroups  with applications to Ramsey theory \cite{KumarRamsey}. Interestingly, the property of being countable infinite and compact simultaneously  is a purely hypergroups theoretical property as any infinite compact group can never be countable. We also obtain the following analogue of  the Hardy-Littlewood inequality for this class of hypergroups $H_a$. 
 
 	The Hardy-Littlewood inequality is obtained by  the following Paley-type inequality for compact hypergroups. 
	\begin{thm}  Let $K$ be a compact hypergroup and let $1<p \leq 2.$ If $\varphi(\pi)$ is a positive sequence over $\widehat{K}$ such that the quantity
		\begin{equation} 
		M_\varphi:= \sup_{y>0} y \sum_{\overset{\pi \in \widehat{K}}{\varphi(\pi) \geq y}} k_\pi^2 
		\end{equation} is finite, then we have 
		\begin{equation} 
		\left(\sum_{\pi \in \widehat{K}} k_\pi^2 \left( \frac{\|\widehat{f}(\pi)\|_{\textnormal{HS}}}{\sqrt{k_\pi}} \right)^p \varphi(\pi)^{2-p} \right)^{\frac{1}{p}} \lesssim M_\varphi^{\frac{2-p}{p}} \|f\|_{L^p(K)}. 
		\end{equation}
	\end{thm}
	
	The Paley-type inequality describes the growth of the Fourier transform of a function in terms of its $L^p$-norm. Interpolating the Paley-inequality with the Hausdorff-Young inequality one can obtain the following H\"ormander's version of the  Hausdorff-Young-Paley inequality,
	\begin{equation}\label{3}
	\left(\int\limits_{\mathbb{R}^n}|(\mathscr{F}f)(\xi)\phi(\xi)^{ \frac{1}{r}-\frac{1}{p'} }|^r\, d \xi \right)^{\frac{1}{r}}\leq \Vert f \Vert_{L^p(\mathbb{R}^n)},\,\,\,1<p\leq r\leq p'<\infty, \,\,1<p<2.
	\end{equation} Also, as a consequence  of the Hausdorff-Young-Paley inequality, H\"ormander \cite[page 106]{Hormander1960} proves that the condition 
	\begin{equation}\label{4}
	\sup_{t>0}t^b\{\xi\in \mathbb{R}^n:m(\xi)\geq t\}<\infty,\quad \frac{1}{p}-\frac{1}{q}=\frac{1}{b},
	\end{equation}where $1<p\leq 2\leq q<\infty,$ implies the existence of a bounded extension of a Fourier multiplier $T_{m}$ with symbol $m$ from $L^p(\mathbb{R}^n)$ to $L^q(\mathbb{R}^n).$ Recently, the second author with his collaborator R. Akylzhanov  extended H\"ormander's classical results to unimodular locally compact groups and to homogeneous spaces \cite{ARN, AR}. In \cite{AR}, the key idea behind the extension of H\"ormander theorem is the reformulation of this theorem as follows: 
	$$\|T_m\|_{L^p(\mathbb{R}^n) \rightarrow L^q(\mathbb{R}^n)} \lesssim \sup_{s>0} s\left( \int_{ \{\xi \in \mathbb{R}^n :\, m(\xi) \geq s\} } d\xi \right)^{\frac{1}{p}-\frac{1}{q}} \simeq \|m\|_{L^{r, \infty}(\mathbb{R}^n)} \simeq \|T_m\|_{L^{r, \infty}(\textnormal{VN}(\mathbb{R}^n))},$$ where $\frac{1}{r}=\frac{1}{p}-\frac{1}{q},$ $\|m\|_{L^{r, \infty}(\mathbb{R}^n)}$ is the Lorentz norm of $m,$ and $\|T_m\|_{L^{r, \infty}(\textnormal{VN}(\mathbb{R}^n))}$ is the norm of the operator $T_m$ in the Lorentz space on the group von Neumann algebra $\textnormal{VN}(\mathbb{R}^n)$ of $\mathbb{R}^n.$ Then one can use the Lorentz spaces and group von Neumann algebra technique for extending it to general locally compact unimodular groups. The unimodularity assumption has its own advantages such as existence of the canonical trace on the group von Neumann algebra and consequently,  Plancherel formula and the Hausdorff-Young inequality. It was also pointed out that the unimodularity can be avoided by using the Tomita-Takesaki modular theory and the Haagerup reduction technique.
	
	By interpolating the Hausdorff-Young inequality and Paley-type inequality we get the following Hausdorff-Young-Paley inequality for compact hypergroups.
	
	\begin{thm}
		Let $K$ be a compact hypergroup and let $1<p \leq b \leq p'<\infty.$ If a positive sequence $\varphi(\pi), \pi \in \widehat{K},$ satisfies the condition 
		\begin{equation}
		M_\varphi:= \sup_{y>0} y \sum_{\overset{\pi \in \widehat{K}}{\varphi(\pi) \geq y}} k_\pi^2 <\infty
		\end{equation}  then we have
		\begin{equation}
		\left(\sum_{\pi \in \widehat{K}} k_\pi^2 \left( \frac{\|\widehat{f}(\pi)\|_{\textnormal{HS}}}{\sqrt{k_\pi}}\varphi(\pi)^{\frac{1}{b}-\frac{1}{p'}} \right)^b  \right)^{\frac{1}{b}} \lesssim M_\varphi^{\frac{1}{b}-\frac{1}{p'}} \|f\|_{L^p(K)}. 
		\end{equation}
	\end{thm}

	Throughout the paper, we denote by $\mathbb{N}$ the set of natural numbers and set $\mathbb{N}_0=\mathbb{N} \cup \{0\}.$ For notational convenience, we take empty sums to be zero.

	\section{Preliminaries}
	For the basics of  compact hypergroups one can refer to standard books, monographs and research papers \cite{Dunkl, Jewett, Bloom, Ken, Ken2,  VremD, Vrem}. However we mention here certain results we need.
	\subsection{Definitions and representations of compact hypergroups} In \cite{Jewett}, Jewett refers to hypergroups as convos. We begin this section with the definition of a compact hypergroup.
	\begin{defn} A compact {\it hypergroup} is a non empty compact Hausdorff space  $K$  with a weakly continuous, associative convolution $*$ on the Banach space $M(K)$ of all bounded regular Borel measures on $K$ such that $(M(K), *)$ becomes a Banach algebra and the following properties hold: 
		\begin{enumerate}
			\item[(i)] For any $x,y \in K,$ the convolution $\delta_x*\delta_y$ is a probability measure with compact support, where $\delta_x$ is the point mass measure at $x.$ Also, the mapping $(x,y)\mapsto \spt(\delta_x*\delta_y)$ is continuous from $K\times K$ to the space  $\mathcal{C}(K)$ of all nonempty compact subsets of $K$ equipped with the Michael (Vietoris) topology (see \cite{mi} for details).  
			\item[(ii)]  There exists a unique element $e \in K $ such that $\delta_x*\delta_e=\delta_e*\delta_x=\delta_x$ for every $x\in K.$
			\item[(iii)] There is a homeomophism $x \mapsto \check{x}$ on $K$ of order two which induces an involution on $M(K)$ where  $\check{\mu}(E)= \mu (\check{E})$ for any Borel set $E,$ and  $e \in \spt(\delta_x*\delta_y)$ if and only if $x = \check{y}.$ 
		\end{enumerate}
	\end{defn}
	Note that the weak continuity assures that the convolution of bounded measures on a hypergroup is uniquely determined by the convolution of point measures. A compact hypergroup is called  a commutative compact hypergroup if the convolution is commutative. A compact hypergroup $K$ is called {\it hermitian} if the involution on $K$ is the identity map, i.e., $\check{x}=x$ for all $x \in K.$ 
	Note that a hermitian hypergroup is commutative. Every compact group is  a trivial example of a compact hypergroup. Other essential and non-trivial examples are double coset hypergroups $G//H$ asing from a Gelfand pair $(G, H)$ for a  compact group $G$ and a closed subgroup $H$ \cite{Jewett}, conjugacy classes of  compact Lie groups \cite{Vrem,Bloom}, countable compact hypergroups \cite{dun2, Bloom}, Jacobi hypergroups \cite{Gasper4, Bloom}, hypergroup joins \cite{Vrem2} of compact hypergroups by  finite hypergroups \cite{Alaga,Bloom}.

	A {\it left Haar measure} $\lambda$ on $K$ is a non-zero positive Radon measure such that 
	$$\int_K f(x*y) d\lambda(y)=\int_K f(y)\, d\lambda(y)\quad (\forall x \in K, \, f \in C_c(K)),$$ where we used the notation $f(x*y)=(\delta_x*\delta_y)(f)$. It is well known that a Haar measure is unique if it exists \cite{Jewett}.  Throughout this article, a left Haar measure is simply called a Haar measure.
	We would like to make a remark here that it still not known if a general hypergroup has a Haar measure but several important class of hypergroups including commutative hypergroups, compact hypergroups, discrete hypergroups, nilpotent hypergroups possess a Haar measure \cite{Jewett, Bloom, Willson, Amini}.
	
	An {\it irreducible representation} $\pi$ of $K$ is
	an irreducible $*$- algebra representation of $M(K)$ into  $\mathcal{L}(\mathcal{H}_\pi),$ the algebra of all bounded linear operators on some Hilbert space $\mathcal{H}_\pi,$ such that
	\begin{itemize}
		\item[(i)] $\pi(\delta_e)=I$ and 
		\item[(ii)] for every $u,v\in\mathcal{H}_\pi,$ the mapping $\mu\mapsto\langle \pi(\mu)u,v \rangle$ is continuous from $M(K)^+$ to $\mathbb{C},$ where $M(K)^+$ is equipped with the weak (cone) topology.
	\end{itemize} 
		In \cite{Jewett} it was also included in the definition  of a representation that $\pi$ must be norm decreasing, that is, $\|\pi(\mu)\|_{\text{op}} \leq \|\mu\|,$ but it follows as a consequence of the above definition. For any $x \in K,$ we also write $\pi(\delta_x)$ as $\pi(x).$ Therefore, we get $ \|\pi(x)\|_{\text{op}} \leq \|\delta_x\|=1,$ where $\|\cdot\|_{\text{op}}$ denotes the operator norm on  $\mathcal{L}(\mathcal{H}_{\pi}).$ 
	\subsection{Fourier analysis on compact hypergroups}
	Let $K$ be  a compact hypergroup with the normalized Haar measure $\lambda$ and let $ \widehat{K}$  be the set of irreducible inequivalent continuous representations of $K.$ Throughout this paper we will assume that $K$ is metrizable which is equivalent the condition that $\widehat{K}$ is countable \cite{FR}.   The set $\widehat{K}$ equipped with the discrete topology is called the dual space of $K$. Vrem \cite{Vrem} showed that every irreducible representation $(\pi, \mathcal{H}_{\pi})$ of a compact hypergroup is finite dimensional. For any $\pi \in \widehat{K},$ the map $x \mapsto \langle \pi(x)u, v \rangle$ for $u, v \in \mathcal{H}_{\pi}$ is called matrix coefficient function and is denoted by $\pi_{u,v}.$ Let $\pi(x)= [\pi_{i,j}]_{d_{\pi} \times d_{\pi}}$ be the matrix representation of any $(\pi, \mathcal{H}_{\pi})$ of dimension $d_{\pi}$ with respect to an orthonormal basis $\left\lbrace e_i \right\rbrace_{i=1}^{d_{\pi}}$ of $\mathcal{H}_{\pi}.$ For each pair $\pi, \pi \in \widehat{K}$ there exists a constant $k_{\pi} \geq d_{\pi}$ such that 
	\begin{equation} \label{ortho}
	\int_K \pi_{i,j}(x) \overline{\pi'_{m,l}(x)}\, d\lambda(x) = \begin{cases} \frac{1}{k_{\pi}} & \text{when}\, i=m, j=l, \, \text{and}\, \pi = \pi', \\ 0 & \text{otherwise}.
	\end{cases}
	\end{equation}
	If $K$ is a compact group then $k_{\pi} = d_{\pi}$ \cite [Theorem 2.6]{Vrem}. The  constant $k_\pi$ is called the hyperdimension of the representation $\pi$ \cite{Alaga}. The function $x \mapsto \chi_\pi(x)=:\text{Tr}(\pi(x))$ is called (hypergroup) {\it character} and it is a continuous function.  The following relation for characters can be derived from orthogonality relation \eqref{ortho} of matrix coefficients
	\begin{equation} \label{orch}
	\int_K \chi_{\pi}(x) \overline{\chi_{\pi'}(x)} d\lambda(x) = \begin{cases} \frac{d_\pi}{k_\pi} & \text{if} \,\,\pi =\pi', \\ 0 & \text{otherwise}, \end{cases}
	\end{equation}
	for all $\pi, \pi' \in \widehat{K}.$ Therefore, $\|\chi_\pi\|_{L^2(K)}^2 =\frac{d_\pi}{k_\pi}.$
	
	The $\ell^p_{\text{sch}}$-spaces on $\widehat{K}$ can be defined as $\ell^p_{\text{sch}}(\widehat{K})$ defined in \cite[D.37, D. 36(e)]{HR}. These spaces are studied by Vrem \cite{VremD}. First, for the  space of Fourier coefficients of functions on $K$ we set 
	\begin{equation}
	\Sigma(K)=\{\sigma: \pi \mapsto \sigma(\pi) \in \mathbb{C}^{d_\pi \times d_\pi}: \, \pi \in \widehat{K}\} = \prod_{\pi \in\widehat{K}} \mathbb{C}^{d_\pi \times d_\pi}.
	\end{equation} 
	  The space $\ell^p_{\text{sch}}(\widehat{K}) \subset \Sigma(K)$ is defined by the norm 
	\begin{equation}
	\|\sigma\|_{\ell^p_{\text{sch}}(\widehat{K})}:= \left(\sum_{\pi \in \widehat{K}} k_\pi  \|\sigma(\pi)\|_{S^p}^p \right)^{\frac{1}{p}},\,\,\, \sigma \in \Sigma(K),\,\, 1\leq p<\infty,
	\end{equation}
	and
	$$\|\sigma\|_{\ell^\infty_{\text{sch}}(\widehat{K})}:= \sup_{\pi \in \widehat{K}} \|\sigma(\pi)\|_{\mathcal{L}(\mathcal{H}_\pi)}\,\,\,\, \sigma \in \Sigma(K). $$
	
	The set of all $\sigma \in \Sigma(K)$ such that $\#\{\pi \in \widehat{K}:\, \sigma(\pi) \neq 0\}<\infty$ denoted by $\Sigma_c(\widehat{K})$ and $\Sigma_0(K)$ is the set of all $\sigma \in \Sigma(K)$ such that $\#\{\pi \in \widehat{K}:\, \|\sigma(\pi)\|_{\mathcal{L}(\mathcal{H}_\pi)} \geq \epsilon \}<\infty$ for all $\epsilon>0.$ For each $\pi \in \widehat{K},$ the Fourier transform $\widehat{f}$ of $f \in L^1(K)$ is defined as
	$$\widehat{f}(\pi) = \int_K f(x) \bar{\pi}(x)\, d\lambda(x),$$ where $\bar{\pi}$ is the conjugate representation of $\pi.$  
	  Vrem \cite{Vrem} proved that the map $f \mapsto \widehat{f}$ is a non norm increasing $*$-isomorphism of $L^1(K)$ onto a dense subalgebra of $\Sigma_0(K).$ 
	 For $f \in L^2(K),$  we have 
	\begin{equation} \label{Fseries}
	f = \sum_{\pi \in \widehat{K}} k_{\pi} \sum_{i,j=1}^{d_{\pi}} \widehat{f}(\pi)_{i,j} \pi_{i,j}
	\end{equation}  
	and the series converges in $L^2(K)$ \cite[Corollary 2.10]{Vrem}. Hence, we have the following Plancherel identity $$\|f\|_2^2= \sum_{\pi \in \widehat{K}} k_{\pi} \sum_{i,j=1}^{d_{\pi}} |\widehat{f}(\pi)_{i,j}|^2 = \sum_{\pi \in \widehat{K}} k_{\pi} \|\widehat{f}(\pi)\|_{\textnormal{HS}}^2 =\|\widehat{f}\|_{\ell^2_{\text{sch}}(\widehat{K})}^2. $$ 
	
	The following Hausdorff-Young inequality holds for Fourier transform on compact hypergroups \cite{VremD}.   
	\begin{thm} \label{HYsch}
		Let $1 \leq p \leq 2$ with $\frac{1}{p}+\frac{1}{p'}=1.$ For any $f \in L^p(K)$ we have the following inequality 
		\begin{equation} \label{HY1sch}
		\left( \sum_{\pi \in \widehat{K}} k_\pi \|\widehat{f}(\pi)\|^{p'}_{S^p} \right)^{\frac{1}{p'}} =      \|\widehat{f}\|_{\ell^{p'}_{\text{sch}}(\widehat{K})} \leq \|f\|_{L^p(K)}.
		\end{equation}
	\end{thm}
	
	Recently, the first author with R. Sarma \cite{KR} also obtained Hausdorff-Young inequality using different norm which was useful to study Hausdorff-Young inequality for Orlicz spaces \cite{KR}. We will discuss it in the next section in more details.
	
	\subsection{Commutative compact hypergroups} In this section we assume  that compact hypergroup $K$ is commutative. Then every representation of $K$ is one dimensional. The dual space  of $K$ defined as follows $$ \widehat{K}= \left\lbrace\chi \in C^b(K): \chi \neq 0,\,\chi(\check{m})= \overline{\chi(m)},\, (\delta_m*\delta_n)(\chi)= \chi(m) \chi(n) \, \text{for all}\,\, m,n \in K \right\rbrace.$$  
	
	An  element in $\widehat{K}$ will be called a {\it character}. Equip  $\widehat{K}$ with the uniform convergence on the compact sets. In case of a compact hypergroups $K$ the dual space $\widehat{K}$ is discrete. In general, $\widehat{K}$ may not have a dual hypergroup structure with respect to the pointwise product \cite[Example 9.1 C]{Jewett} but it holds for most ``natural" hypergroups including the conjugacy classes of compact groups. Then the Fourier transform on $L^1(K, \lambda)$ is defined by 
	$$\widehat{f}(\chi):= \int_K f(x)\, \overline{\chi(x)}\, d\lambda(x),\quad \chi \in \widehat{K}.$$
	The Fourier transform is injective and there exists a Radon measure $\omega$ on $\widehat{K},$ called the Plancherel measure on $\widehat{K}$ such that the map $f \mapsto \widehat{f}$ extends to an isometric isomorphism from $L^2(K, \lambda)$ onto $L^2(\widehat{K}, \omega),$ that is, 
	\begin{equation} \label{pabel}
	\sum_{\chi \in \widehat{K}} |\widehat{f}(\chi)|^2 d\omega(\chi)= \int_K |f(x)|^2\, d\lambda(x).
	\end{equation} 
	In this case, the Fourier series of $f$ given by \eqref{Fseries} takes the form
	\begin{equation}
	f= \sum_{\chi \in \widehat{K}} k_{\chi}\, \widehat{f}(\chi)\, \chi.
	\end{equation}
	It follows from the orthogonality relation of characters \eqref{orch} that the set $\{k_\chi^{\frac{1}{2}} \chi\}_{\chi \in \widehat{K}}$ forms an orthonormal basis of $L^2(K).$ It is also known that for each $\chi \in \widehat{K}$ we have that $\omega(\chi)=k_{\chi}$ \cite[Proposition 1.2]{Alaga}. If $K$ is a compact commutative group then $k_\chi=d_\chi=1$ for all $\chi \in \widehat{K};$ and therefore Plancherel measure on $\widehat{K}$ is constant $1.$    
	
	\section{Hausdorff-Young-Paley and  Hardy-Littlewood inequalities on compact hypergroups}
	In this section, we will study Paley inequality, Hausdorff-Young-Paley inequality and Hardy Littlewood inequality for compact hypergroups. At times, we will denote $L^p(K, \lambda)$ by $L^p(K)$ for simplicity. 
	\subsection{Paley inequality on compact hypergroups} In this subsection, we prove Paley inequality for compact hypergroups. Paley inequality is an important inequality in itself but also plays a vital role to obtain Hardy-Littlewood inequality and Hausdorff-Young-Paley inequality for compact hypergroups. 
	\begin{thm} \label{Paley} Let $K$ be a compact hypergroup and let $1<p \leq 2.$ If $\varphi:\widehat{K} \rightarrow (0, \infty)$ is a function such that
		\begin{equation}\label{Paleycondi}
		M_\varphi:= \sup_{y>0} y \sum_{\overset{\pi \in \widehat{K}}{\varphi(\pi) \geq y}} k_\pi^2<\infty. 
		\end{equation} Then, for all $f \in L^p(K),$ we have 
		\begin{equation} \label{Paley1}
		\left(\sum_{\pi \in \widehat{K}} k_\pi^2 \left( \frac{\|\widehat{f}(\pi)\|_{\textnormal{HS}}}{\sqrt{k_\pi}} \right)^p \varphi(\pi)^{2-p} \right)^{\frac{1}{p}} \lesssim M_\varphi^{\frac{2-p}{p}} \|f\|_{L^p(K)}. 
		\end{equation}
	\end{thm}
	\begin{proof} 
		Let us consider the measure on $\nu$ on the dual space $\widehat{K}$ on $K$ given by 
		$$\nu(\{\pi\})= \varphi(\pi)^2 k_{\pi}^2,\,\,\,\,\, \pi \in \widehat{K}. $$ 
		Define the space $L^p(\widehat{K}, \nu), 1 \leq p <\infty,$ as the space of all real or complex sequences $a: \pi  \mapsto a_\pi $ such that 
		$$\|a\|_{L^p(\widehat{K}, \nu)}= \left( \sum_{\pi \in \widehat{K}} |a_\pi|^p \varphi(\pi)^2 k_{\pi}^2  \right)^{\frac{1}{p}}<\infty.$$
		We will show that the sublinear operator $A:L^p(K, \lambda) \rightarrow L^p(\widehat{K}, \nu)$ defined by 
		$$Af:= \left( \frac{\|\widehat{f}(\pi)\|_{\textnormal{HS}}}{\sqrt{k_\pi}\, \varphi(\pi)}\right)_{\pi \in \widehat{K}}$$ is well defined and bounded for $1<p \leq 2.$ In other words, we will get the following estimate which will eventually give us the required estimate \eqref{Paley1},
		\begin{equation} \label{vis}
		\|Af\|_{L^p(\widehat{K}, \nu)} = \left( \sum_{\pi \in \widehat{K}} \left( \frac{\|\widehat{f}(\pi)\|_{\textnormal{HS}}}{\sqrt{k_\pi} \varphi(\pi)} \right)^p \varphi(\pi)^2 k_\pi^2 \right)^{\frac{1}{p}} \lesssim M_\varphi^{\frac{2-p}{p}} \|f\|_{L^p(K)},
		\end{equation} 
		where $M_\varphi:= \sup_{y>0}  y \sum_{\overset{\pi \in \widehat{K}}{\varphi(\pi) \geq y}} k_\pi^2. $
		To prove the above estimate \eqref{vis} it is enough to show that $A$ is weak type $(1,1)$ and weak type $(2,2),$ thanks to Marcinkiewicz interpolation theorem. In fact, we show that, with the distribution function $\nu_{\widehat{K}},$ that 
		\begin{equation} \label{Vish11}
		\nu_{\widehat{K}}(y; Af) \leq \frac{M_1 \|f\|_{L^1(K)}}{y} \,\,\,\, \text{with the norm}\,\, M_1= M_\varphi,
		\end{equation}
		\begin{equation} \label{Vish22}
		\nu_{\widehat{K}}(y; Af) \leq \left( \frac{M_2 \|f\|_{L^2(K)}}{y} \right)^{2}\,\,\,\, \text{with the norm}\,\, M_2=1,
		\end{equation}
		where $\nu_{\widehat{K}}(y; Af)$ is defined by $\nu_{\widehat{K}}(y; Af):= \sum_{\overset{\pi \in \widehat{K}}{|(Af)(\pi)| \geq y}} \nu(\pi), \,\,\, y>0.$
		
		First, we show that $A$ is of type $(1,1)$ with norm $M_1=M_\varphi;$ more precisely we show that 
		\begin{equation} \label{vis8}
		\nu_{\widehat{K}}(y; Af)= \nu \left\{ \pi \in \widehat{K}: \frac{\|\widehat{f}(\pi)\|_{\textnormal{HS}}}{\sqrt{k_\pi} \varphi(\pi)}>y \right\} \lesssim \frac{M_\varphi \|f\|_{L^1(K)}}{y},
		\end{equation} where $\nu \left\{ \pi \in \widehat{K}: \frac{\|\widehat{f}(\pi)\|_{\textnormal{HS}}}{\sqrt{k_\pi} \varphi(\pi)}>y \right\}$ can be interpreted as the weighted sum $\sum \varphi(\pi)^2 k_\pi^2$ taken over those $\pi \in \widehat{K}$ such that $\frac{\|\widehat{f}(\pi)\|_{\textnormal{HS}}}{\sqrt{k_\pi} \varphi(\pi)}>y.$
		By the defintion of the Fourier transform and the fact that $\pi$ is a norm decreasing $*$-homomorphism, i.e., $\|\pi(\check{x})\|_{\text{op}} \leq 1$ for all $x \in K,$ we have 
		$$\|\widehat{f}(\pi)\|_{\textnormal{HS}} \leq  \|f\|_{L^1(K)} \|\pi(\check{x})\|_{\textnormal{HS}} \leq \|f\|_{L^1(K)}  \sqrt{d_\pi} \|\pi(\check{x})\|_{\text{op}} \leq \sqrt{d_\pi}  \|f\|_{L^1(K)}.$$
		Therefore, by using  $d_\pi \leq k_\pi,$ we get
		$$y < \frac{\|\widehat{f}(\pi)\|_{\textnormal{HS}}}{\sqrt{k_\pi} \varphi(\pi)} \leq \frac{\sqrt{d_\pi}  \|f\|_{L^1(K)}}{\sqrt{k_\pi} \varphi(\pi)} \leq \frac{\|f\|_{L^1(K)}}{ \varphi(\pi)}.$$
		
		This inequality yields that 
		$$\left\{ \pi \in \widehat{K}: \frac{\|\widehat{f}(\pi)\|_{\textnormal{HS}}}{\sqrt{k_\pi} \varphi(\pi)}>y \right\} \subset \left\{ \pi \in \widehat{K}: \frac{\|f\|_{L^1(K)}}{ \varphi(\pi)} >y \right\}$$ for any $y >0.$ So  
		$$ \nu \left\{ \pi \in \widehat{K}: \frac{\|\widehat{f}(\pi)\|_{\textnormal{HS}}}{\sqrt{k_\pi} \varphi(\pi)}>y \right\} \leq \nu \left\{ \pi \in \widehat{K}: \frac{\|f\|_{L^1(K)}}{ \varphi(\pi)} >y \right\}.$$
		Setting $w= \frac{\|f\|_{L^1(K)}}{y},$ we have 
		$$\nu \left\{ \pi \in \widehat{K}: \frac{\|\widehat{f}(\pi)\|_{\textnormal{HS}}}{\sqrt{k_\pi} \varphi(\pi)}>y   \right\} \leq \sum_{\overset{\pi \in \widehat{K}}{\varphi(\pi) \leq w}} \varphi(\pi)^2 k_\pi^2. $$
		We claim that 
		\begin{equation}
		\sum_{\overset{\pi \in \widehat{K}}{\varphi(\pi) \leq w}} \varphi(\pi)^2 k_\pi^2 \lesssim M_\varphi w. 
		\end{equation}
		In fact, we have 
		$$\sum_{\overset{\pi \in \widehat{K}}{\varphi(\pi) \leq w}} \varphi(\pi)^2 k_\pi^2  = \sum_{\overset{\pi \in \widehat{K}}{\varphi(\pi) \leq w}} k_\pi^2 \int_{0}^{\varphi^2(\pi)} d\tau.$$
		By interchanging  sum and integration we have 
		$$\sum_{\overset{\pi \in \widehat{K}}{\varphi(\pi) \leq w}} k_\pi^2 \int_{0}^{\varphi^2(\pi)} d\tau = \int_{0}^{w^2} d \tau \sum_{\overset{\pi \in \widehat{K}}{\tau^{\frac{1}{2}} \leq \varphi(\pi) \leq w}} k_\pi^2.$$ 
		Next, by making substitution $\tau= t^2$ it yields to 
		$$\int_{0}^{w^2} d \tau \sum_{\overset{\pi \in \widehat{K}}{\tau^{\frac{1}{2}} \leq \varphi(\pi) \leq w}} k_\pi^2 = 2 \int_{0}^w t dt \sum_{\overset{\pi \in \widehat{K}}{ t \leq \varphi(\pi) \leq w}} k_\pi^2 \leq 2 \int_0^w t \, dt \sum_{\overset{\pi \in \widehat{K}}{ t \leq \varphi(\pi)}} k_\pi^2.$$
		
		Since $$t \sum_{\overset{\pi \in \widehat{K}}{ t \leq \varphi(\pi)}} k_\pi^2 \leq \sup_{t>0} t \sum_{\overset{\pi \in \widehat{K}}{ t \leq \varphi(\pi)}} k_\pi^2=M_\varphi$$ is finite by the assumption, we get 
		$$2 \int_0^w t \, dt \sum_{\overset{\pi \in \widehat{K}}{ t \leq \varphi(\pi)}} k_\pi^2 \lesssim M_\varphi w.$$
		Therefore, we get the required estimate \eqref{vis8}
		$$\nu_{\widehat{K}}(y; Af)= \nu \left\{ \pi \in \widehat{K}: \frac{\|\widehat{f}(\pi)\|_{\textnormal{HS}}}{\sqrt{k_\pi} \varphi(\pi)}>y \right\} \lesssim \frac{M_\varphi \|f\|_{L^1(K)}}{y}.$$
		
		Now, we will prove that $A$ is weak type $(2,2),$ that is, the equality \eqref{Vish22}. By using Plancherel's identity we get 
		\begin{align*}
		y^2 \nu_{\widehat{K}}(y; Af) \leq \|Af\|_{L^2(\widehat{K}, \nu)}^2 &= \sum_{\pi \in \widehat{K}} k_\pi^2 \left( \frac{\|\widehat{f}(\pi)\|_{\textnormal{HS}}}{\sqrt{k_\pi} \varphi(\pi)} \right)^2 \varphi(\pi)^2 \\&= \sum_{\pi \in \widehat{K}} k_\pi \|\widehat{f}(\pi)\|_{\textnormal{HS}}^2 =  \|f\|_{L^2(K)}^2. 
		\end{align*}Thus $A$ is of type $(2,2)$ with norm $M_2 \leq 1.$ Thus we have proved \eqref{Vish22} and \eqref{Vish11}. Thus, by using the Marcinkiewicz interpolation theorem with $p_1=1,\, p_2=2$ and $\frac{1}{p}=1- \theta+\frac{\theta}{2}$ we now obtain 
		$$  \left( \sum_{\pi \in \widehat{K}} \left( \frac{\|\widehat{f}(\pi)\|_{\textnormal{HS}}}{\sqrt{k_\pi} \varphi(\pi)} \right)^p \varphi(\pi)^2 k_\pi^2 \right)^{\frac{1}{p}} = \|Af\|_{L^p(\widehat{K}, \nu)} \lesssim       M_\varphi^{\frac{2-p}{p}} \|f\|_{L^p(K)}.$$ This completes the proof.
	\end{proof}

   \begin{rem}
   	One may notice that instead on Schatten $p$-norm we used Hilbert-Schimdt norm in Theorem \ref{Paley}. This is because Hilbert-Schmidt norm gives sharp inequality in Paley-type theorem as already noticed in \cite{ARN} for compact homogeneous spaces and in \cite{Youn} for compact quantum groups. We will see this for compact hypergroups from the discussion below. 
   \end{rem}
	
	Now, we will define and discuss an another important family of Lebesgue spaces $\ell^p$ on $\widehat{K}$ defined using the  Hilbert-Schmidt norm $\|\cdot\|_{\textnormal{HS}}$  instead of Schatten $p$-norm $\|\cdot\|_{S^p}$  on the space of $(d_\pi \times d_\pi)$-dimensional matrices. 
	Recently, these $\ell^p$-space have been studied in more details by the second author and his collaborators in the context of compact Lie groups and compact homogeneous spaces \cite{RuzT, ARN1,AMR,KM,FR,AR14}. In particular, it was shown in \cite{AR14} that the space $\ell^p(\widehat{G})$ and the Hausdorff-Young inequality for it become useful for convergence of Fourier series and characterization of Gevrey-Roumieu ultradifferentiable functions and Gevrey-Beurling ultradifferentiable functions on compact homogeneous manifolds. 
	
	 Next, we define the Lebesgue spaces $\ell^p(\widehat{K}) \subset \Sigma(K)$ by the condition
	\begin{equation}
	\|\sigma\|_{\ell^p(\widehat{K})}:= \left( \sum_{\pi \in \widehat{K}}  k_\pi^{(2-\frac{p}{2})} \|\sigma(\pi)\|_{\textnormal{HS}}^p \right)^{\frac{1}{p}},\,\,\,\, 1 \leq p <\infty,
	\end{equation}
	and $$\|\sigma\|_{\ell^\infty(\widehat{K})}:= \sup_{\pi \in \widehat{K}} k_\pi^{-\frac{1}{2}} \|\sigma(\pi)\|_{\textnormal{HS}}.$$ 
	
	
	We note here that for compact groups such spaces were introduced in \cite[Chapter 10]{RuzT}.
	
	The following proposition presents the relation between both norms on Lebesgue spaces on $\widehat{K}.$  
	\begin{prop} \label{Estisch}
		For $1 \leq p \leq 2,$ we have the following  continuous embeddings as  well as the estimates:  
		$\ell^p(\widehat{K}) \hookrightarrow \ell^p_{sch}(\widehat{K})$ and $\|\sigma\|_{\ell^p_{sch}(\widehat{K})} \leq \|\sigma\|_{\ell^p(\widehat{K})}$ \, for all $\sigma \in \Sigma(K).$ For $2 \leq p \leq \infty,$ we have 
		$\ell^p_{sch}(\widehat{K}) \hookrightarrow \ell^p(\widehat{K})  $ and $ \|\sigma\|_{\ell^p(\widehat{K})} \leq  \|\sigma\|_{\ell^p_{sch}(\widehat{K})} $ \, for all $\sigma \in \Sigma(K).$
	\end{prop}
	\begin{proof}
		For $p=2,$ the norms coincide since $S^2=\textnormal{HS}.$ Let $1 \leq p <2.$ Since $\sigma(\pi) \in \mathbb{C}^{d_\pi \times d_\pi},$ denoting $s_j$ its singular number, by H$\ddot{\text{o}}$lder inequality we have 
		\begin{equation} \label{Scht}
		\|\sigma(\pi)\|_{S^p}^p = \sum_{j=1}^{d_\pi} s_j^p \leq \left(  \sum_{j=1}^{d_\pi} 1 \right)^{\frac{2-p}{2}} \left( \sum_{j=1}^{d_\pi} s_{j}^{p \frac{2}{p}} \right)^{\frac{p}{2}} = d_{\pi}^{\frac{2-p}{2}} \|\sigma(\pi)\|_{\textnormal{HS}}^p.
		\end{equation}
		Consequently, it follows that 
		$$\|\sigma\|^p_{\ell_{sch}^p(\widehat{K})}= \sum_{\pi \in \widehat{K}} k_\pi \|\sigma(\pi)\|_{S^p}^p \leq \sum_{\pi \in \widehat{K}} k_\pi d_{\pi}^{\frac{2-p}{2}} \|\sigma(\pi)\|_{\textnormal{HS}}^p \leq \sum_{\pi \in \widehat{K}} k_\pi k_{\pi}^{\frac{2-p}{2}} \|\sigma(\pi)\|_{\textnormal{HS}}^p = \|\sigma\|_{\ell^p(\widehat{K})}^p.$$
		Now, for $2<p <\infty,$ we have 
		\begin{equation}
		\|\sigma(\pi)\|_{\textnormal{HS}}^2 = \sum_{j=1}^{d_\pi} s_j^2 \leq \left( \sum_{j=1}^{d_\pi} 1 \right)^{\frac{p-2}{p}} \left( \sum_{j=1}^{d_\pi} s_j^{2 \frac{p}{2}}\right)^{\frac{2}{p}} = d_\pi^{\frac{p-2}{p}} \|\sigma(\pi)\|_{S^p}^2,
		\end{equation} 
		implying 
		$$ \|\sigma(\pi)\|_{\textnormal{HS}} \leq d_\pi^{\frac{p-2}{2p}} \|\sigma(\pi)\|_{S^p}.$$
		Therefore, we have 
		$$\|\sigma\|_{\ell^p(\widehat{K})}= \sum_{\pi \in \widehat{k}} k_\pi^{(2-\frac{p}{2})} \|\sigma(\pi)\|_{\textnormal{HS}}^p \leq \sum_{\pi \in \widehat{K}}  k_\pi^{(2-\frac{p}{2})} d_\pi^{\frac{p-2}{2}} \|\sigma(\pi)\|^p_{S^p} \leq \sum_{\pi \in \widehat{k}} k_\pi\|\sigma(\pi)\|^p_{S^p}  = \|\sigma\|_{\ell^p_{sch}(\widehat{K})}^p.$$
		Finally, for $p=\infty,$ the inequality $$\|\sigma(\pi)\|_{\textnormal{HS}} \leq k_{\pi}^{\frac{1}{2}}\|\sigma(\pi)\|_{\mathcal{L}(\mathcal{H}_\pi)}$$ implies 
		$$\|\sigma\|_{\ell^\infty(\widehat{K})} = \sup_{\pi \in \widehat{K}} k_\pi^{\frac{1}{2}} \|\sigma(\pi)\|_{\textnormal{HS}} \leq \sup_{\pi \in \widehat{K}} \|\sigma(\pi)\|_{\mathcal{L}(\mathcal{H}_\pi)}= \|\sigma\|_{\ell^\infty_{\text{sch}}(\widehat{K})}.$$
	\end{proof}

The following Hausdorff-Young inequality for Fourier transform on compact hypergroups was recently obtained by the first author and R. Sarma \cite{KR}.  
\begin{thm} \label{HY}
	Let $1 \leq p \leq 2$ with $\frac{1}{p}+\frac{1}{p'}=1.$ For any $f \in L^p(K)$ we have the following inequality 
	\begin{equation} \label{HY1}
	\left( \sum_{\pi \in \widehat{K}} k_\pi^{2-\frac{p'}{2}} \|\widehat{f}(\pi)\|^{p'}_{\textnormal{HS}} \right)^{\frac{1}{p'}} =      \|\widehat{f}\|_{\ell^{p'}(\widehat{K})} \leq \|f\|_{L^p(K)}.
	\end{equation}
\end{thm} 
	In the view of Proposition \ref{Estisch} one can see that the Hausdorff-Young inequality \eqref{HY1sch} using Schatten $p$-norm is sharper than inequality \eqref{HY1}. In \cite{KR}, Theorem \ref{HY} is further used to define Orlicz space on dual of compact hypergroups and to obtained Hausdorff-Young inequality for Orlicz spaces on compact hypergroup.
	
	The  Paley inequality can be reduced  to the familiar form using Schatten $p$-norm. The proof of it is immediate from the inequality \eqref{Scht} and the fact that $d_\pi \leq k_\pi.$ 
	
	\begin{cor} Let $K$ be a compact hypergroup and let $1<p \leq 2.$ If $\varphi:\widehat{K} \rightarrow (0, \infty)$ is a function satisfying condition \eqref{Paleycondi} of Theorem \ref{Paley} then there exist a universal constant $C=C(p)$ such that 
		\begin{equation} 
		\left(\sum_{\pi \in \widehat{G}} k_\pi \|\widehat{f}(\pi)\|^{p}_{S^p}\, \varphi(\pi)^{2-p} \right)^{\frac{1}{p}} \leq C \|f\|_{L^p(K)}. 
		\end{equation}
	\end{cor}

	\subsection{Hardy-Littlewood inequality on compact hypergroups}
	In this section, we apply Paley inequality to get the Hardy-Littlewood inequality on compact hypergroups. This approach has been recently considered to prove the Hardy-Littlewood inequality in the context of compact Lie groups $\text{SU(2)}$ \cite{ARN1}, compact homogeneous manifolds \cite{ARN} and compact quantum groups \cite{AMR,Youn}. The philosophy to derive Hardy-Littlewood inequality is to choose the function $\varphi$ suitably such that condition \eqref{Paleycondi} of Theorem \ref{Paley} is satisfied. In the case of a compact Lie group $G$ of dimension $n$, in \cite{ARN} the authors took $\varphi(\pi)= \langle \pi \rangle^{-n},$ where $\langle \pi \rangle$ denote the eigenvalues of the operator $(1-\Delta_G)^{\frac{1}{2}}$ corresponding to the representation $\pi$ for a Laplacian $\Delta_G$ on $G.$ Although,  for $\text{SU(2)}$  this was proved by repeating the proof of Paley inequality and estimating the bound explicitly (\cite{ARN1}). In the case of compact quantum groups, the proof of this inequality  has been achieved by using the geometric informations of compact quantum groups like spectral triples \cite{AMR} and the natural length function on the dual of compact quantum groups \cite{Youn}. The compact hypergroups in general are not equipped with any geometric and differential structure so we prove the following Hardy-Littlewood inequality for compact hypergroups.  
	\begin{thm} \label{HLG}
	Let $1 <p \leq 2$ and let $K$ be a compact hypergroup. Assume that a positive function $\pi \mapsto \mu_\pi$ on $\widehat{K}$ grows sufficiently fast, that is, 
	\begin{equation} \label{HLGcondi}
	\sum_{\pi \in \widehat{K}} \frac{k_\pi^2}{|\mu_\pi|^\beta}<\infty\,\,\,\, \text{for some}\,\beta \geq 0.
	\end{equation}  Then we have 
	\begin{equation} \label{HL}
	\sum_{\pi \in \widehat{K}} k_\pi^2 |\mu_\pi|^{\beta (p-2)} \left( \frac{\|\widehat{f}(\pi)\|_{\textnormal{HS}}}{\sqrt{k_\pi}} \right)^p \lesssim \|f\|_{L^p(K)}.
	\end{equation}
\end{thm}
\begin{proof}
	By the assumption, we know that 
	$$C:= \sum_{\pi \in \widehat{K}} \frac{k_\pi^2}{|\mu_\pi|^\beta}< \infty.$$ Then we have 
	$$ C \geq \sum_{\overset{\pi \in \widehat{K}}{|\mu_\pi|^\beta \leq \frac{1}{t}}} \frac{k_\pi^2}{|\mu_\pi|^\beta} \geq t \sum_{\overset{\pi \in \widehat{K}}{|\mu_\pi|^\beta \leq \frac{1}{t}}} k_\pi^2 = t \sum_{\overset{\pi \in \widehat{K}}{ \frac{1}{|\mu_\pi|^\beta}} \geq t} k_\pi^2$$ and consequently we have  
	
	$$ \sup_{t>0}t \sum_{\overset{\pi \in \widehat{K}}{ \frac{1}{|\mu_\pi|^\beta}} \geq t} k_\pi^2 \leq C <\infty. $$ Then, as an application of Theorem \ref{Paley} with $\varphi(\pi)= \frac{1}{|\mu_\pi|^\beta},\,\, \pi \in \widehat{K},$ we get the required estimate \eqref{HL}.
\end{proof}
In the case when $K$ is abelian, the Hardy-Littlewood inequality takes the following form. 

\begin{thm} \label{HLabe}
	Let $1<p \leq 2$ and let $K$ be a compact abelian hypergroup. Assume that a positive function $\chi \mapsto \mu_\chi$ on $\widehat{K}$ satisfies the condition
	\begin{equation} \label{HLabecon}
	\sum_{\chi \in \widehat{K}} \frac{k_\chi^2}{|\mu_\chi|^\beta} <\infty\quad \text{for some}\,\, \beta \geq 0.
	\end{equation}
	Then we have 
	\begin{equation}
	\sum_{\chi \in \widehat{K}} k_\chi^{2-\frac{p}{2}} |\mu_\chi|^{\beta(p-2)} |\widehat{f}(\chi)|^p \lesssim \|f\|_{L^p(K)}. 
	\end{equation}
	\end{thm}
\begin{rem}
    We would like to note here that in the case when $K$ is a compact Lie group, the natural choices of $\pi \mapsto \mu_\pi$ is $\pi \mapsto \langle \pi \rangle.$ But for this choice of $ \mu_\pi$ the quantity $\sum_{\pi \in \widehat{K} }\frac{k_\pi}{|\mu_\pi|^{\beta}},$ which turns out to be $\sum_{\pi \in \widehat{K} }\frac{d_\pi}{\langle \pi \rangle^{\beta}}$ in this case, is not finite for $\beta=n:=\dim(G)$ as proved by the second author and Dasgupta \cite{AR14}. So this does not give the Hardy-Littlewood inequality for compact Lie groups, in particular, for $\mathbb{T}^n$ (\cite{ARN}). Surprisingly, the quantity $\sum_{\pi \in \widehat{K} }\frac{k_\pi}{|\mu_\pi|^{\beta}}$ is finite with a natural choice of $\pi \mapsto \mu_\pi$ and $\beta$ for (pure) hypergroups including conjugacy classes of compact Lie groups and countable compact hypergroups as shown in the last section and consequently, provides the Hardy-Littlewood inequality for compact hypergroups.   
    
\end{rem}

	\subsection{Hausdorff-Young-Paley inequality on compact hypergroups}
	In this subsection, we prove the Hausdorff-Young-Paley inequality for compact hypergroups. The Hausdorff-Young-Paley inequality is an important inequality in itself but it serves as an essential tool to prove $L^p$-$L^q$ Fourier multiplier for compact hypergroups. 
	
	The following theorem \cite{BL} is useful  in the proof of the Hausdorff-Young-Paley inequality.  	
	
	\begin{thm}\label{interpolationoperator} Let $d\mu_0(x)= \omega_0(x) d\mu(x),$ $d\mu_1(x)= \omega_1(x) d\mu(x).$ Suppose that $0<p_0, p_1< \infty.$  If a continuous linear operator $A$ admits bounded extensions, $A: L^p(Y,\mu)\rightarrow L^{p_0}(\omega_0) $ and $A: L^p(Y,\mu)\rightarrow L^{p_1}(\omega_1) ,$   then there exists a bounded extension $A: L^p(Y,\mu)\rightarrow L^{b}(\tilde{\omega}) $ of $A$, where  $0<\theta<1, \, \frac{1}{b}= \frac{1-\theta}{p_0}+\frac{\theta}{p_1}$ and 
		$\tilde{\omega}= \omega_0^{\frac{b(1-\theta)}{p_0}} \omega_1^{\frac{b\theta}{p_1}}.$
	\end{thm} 
	
	Now, we are ready to state the Hausdorff-Young-Paley inequality for compact hypergrous. Sometimes, it is also known as Pitt's inequality in the literature on  classical harmonic analysis. 
	
	\begin{thm} \label{HYP}
		Let $K$ be a compact hypergroup and let $1<p \leq b \leq p'<\infty.$ If a function $\varphi: \widehat{K} \rightarrow (0, \infty)$ satisfies the condition 
		\begin{equation}
		M_\varphi:= \sup_{y>0} y \sum_{\overset{\pi \in \widehat{K}}{\varphi(\pi) \geq y}} k_\pi^2 <\infty
		\end{equation}  then we have
		\begin{equation}
		\left(\sum_{\pi \in \widehat{G}} k_\pi^2 \left( \frac{\|\widehat{f}(\pi)\|_{\textnormal{HS}}}{\sqrt{k_\pi}}\varphi(\pi)^{\frac{1}{b}-\frac{1}{p'}} \right)^b  \right)^{\frac{1}{b}} \lesssim M_\varphi^{\frac{1}{b}-\frac{1}{p'}} \|f\|_{L^p(K)}. 
		\end{equation}
	\end{thm}
	\begin{proof} 
		We consider a sublinear operator $A$ which takes a function $f$ to its Fourier coefficient $\widehat{f}(\pi) \in \mathbb{C}^{d_\pi \times d_\pi}$ divided by $\sqrt{k_\pi}$, that is, 
		$$ f \mapsto Af:= \left\{ \frac{\widehat{f}(\pi)}{\sqrt{k_\pi}} \right\}_{\pi \in \widehat{K}},$$
		from $L^p(K)$ into the weighted space $\ell^p(\widehat{K}, \tilde{\omega}).$ The space $\ell^p(\widehat{K}, \tilde{\omega})$ is defined by the norm
		$$\|a\|_{\ell^p(\widehat{K}, \tilde{\omega})}:= \left( \sum_{\pi \in \widehat{K}} \|a(\pi)\|^p_{\textnormal{HS}}\,\, \tilde{\omega}(\pi) \right)^\frac{1}{p},$$ and $\tilde{\omega}$  is a scalar sequence defined on $\widehat{K}$ to be determined. Then the proof of the theorem follows from Theorem \ref{interpolationoperator} if we consider the left hand side of the inequalities \eqref{Paley1} and \eqref{HY1} as $\|Af\|_{\ell^p(\widehat{K}, \tilde{\omega})}$-norm of the operator $A$ in the weighted sequence spaces over $\widehat{K}$ with the weights given by $\omega_0(\pi)= k_\pi^2 \varphi(\pi)^{2-p}$ and $\omega_1(\pi)=k_\pi^2,$ $\pi \in \widehat{K},$ respectively.  
	\end{proof}

	\section{$L^p$-$L^q$-boundedness of Fourier multipliers on compact hypergroups}
	
	In this section, we prove $L^p$-$L^q$ boundedness of Fourier multipliers on compact hypergroups as a natural analogue of H\"ormander's theorem \cite{Hormander1960} on compact hypergroups.
	We will apply the Hausdorff-Young-Paley inequality in Theorem \ref{HYP} to provide a sufficient condition for the  $L^p$-$L^q$ boundedness of Fourier multipliers for the range $1<p \leq 2 \leq q <\infty.$ This approach was developed  by the second author with R. Akylzhanov to prove the $L^p$-$L^q$ boundedness of Fourier multipliers on locally compact groups \cite{AR} by using the von-Neumann algebra machinery. In \cite{NT}, this theorem was proved for the torus $\mathbb{T}$ using a different method. We begin this section by recalling the definition of Fourier multipliers on compact hypergroups.

	An  operator $A$ which is invariant under the left translations will be called a left Fourier multiplier. The left invariant operators can be characterized using the Fourier transform \cite{VremD, SKK}. Indeed, if $A$ is a left Fourier multiplier then there exists a function $\sigma_{A}:\widehat{K} \rightarrow \mathbb{C}^{d_\pi \times d_\pi},$ known as the symbol associated with $A,$ such that 
	          $$\widehat{Af}(\pi)=\sigma_{A}(\pi) \widehat{f}(\pi),\,\,\,\, \pi \in \widehat{K},$$ for all $f$ belonging to a suitable function space on $K.$  
	    In the next result, we show that if the symbol $\sigma_{A}$ of a Fourier multipliers $A$ defined on $C_c(K)$ satisfies certain H\"ormander's condition then $A$ can be extended as a bounded linear operator from $L^p(K)$ to $L^q(K)$ for the range $1<p \leq 2 \leq q <\infty.$           
	
	\begin{thm} \label{Lp-Lqmulti}
		Let $K$ be a compact hypergroup and let $1<p\leq 2 \leq q<\infty.$ Let $A$ be a left Fourier multiplier with symbol $\sigma_{A}$. Then we have 
		\begin{equation} \label{hocon}
		\|A\|_{L^p(K) \rightarrow L^q(K)} \lesssim \sup_{y>0} y \left( \sum_{\overset{\pi \in \widehat{K}}{\|\sigma_A(\pi)\|_{op}\geq y}} k_\pi^2 \right)^{\frac{1}{p}-\frac{1}{q}}.
		\end{equation}
	\end{thm}
	\begin{proof}
		Let us first consider the case when $p \leq q'$ (where $\frac{1}{q}+\frac{1}{q'}=1$). Since $q'\leq 2,$ for $f \in C_c(K),$ the Hausdorff-Young inequality gives \begin{align}
		\|Af\|_{L^q(K)} \leq \|\widehat{Af}\|_{\ell^{q'}(\widehat{K})} &= \|\sigma_A \widehat{f}\|_{\ell^{q'}(\widehat{K})}= \left( \sum_{\pi \in \widehat{K}} k_\pi^2 \left( \frac{\|\sigma_A(\pi) \widehat{f}(\pi)\|_{\textnormal{HS}}}{ \sqrt{k_\pi}} \right)^{q'} \right)^{\frac{1}{q'}}  \\&\leq \left( \sum_{\pi \in \widehat{K}} k_\pi^2  \|\sigma_A(\pi)\|_{\text{op}}^{q'}\left( \frac{ \|\widehat{f}(\pi)\|_{\textnormal{HS}}}{ \sqrt{k_\pi}} \right)^{q'} \right)^{\frac{1}{q'}}.
		\end{align}
		The case $q'\leq p = (p')'$ can be reduced to the case $p \leq q'$ as follows. The $L^p$-duality (see \cite[Theorem 4.2]{ARN1}) yields 
		$$\|A\|_{L^p(K) \rightarrow L^q(K)} = \|A^* \|_{L^{q'}(K) \rightarrow L^{p'}(K)}.$$
		Also, the symbol $\sigma_{A^*}(\pi)$ of the adjoint operator $A^*$ is equal to $\sigma_{A}^*,$ i.e.,  
		$$\sigma_{A^*}(\pi)= \sigma_A(\pi)^*,\,\,\,\,\,\pi \in \widehat{K},$$ and its operator norm $\|\sigma_{A^*}(\pi)\|_{\text{op}}$ is equal to $\|\sigma_A(\pi)\|_{\text{op}}.$ We set $\sigma(\pi)= \|\sigma_A(\pi)\|_{\text{op}}^r I_{d_\pi },  \pi \in \widehat{K},$ where $r= \frac{q-p}{pq},$ and it is easy to see that $$ \|\sigma(\pi)\|_{\text{op}}=\|\sigma_A(\pi)\|_{\text{op}}^r. $$  Now, its time to apply Theorem \ref{HYP}.  We observe that with $\varphi(\pi)= \|\sigma(\pi)\|_{\text{op}}, \,\, \pi \in \widehat{K},$ and $b=q',$ the assumption of Theorem \ref{HYP} is satisfied, and since $\frac{1}{q'}-\frac{1}{p'}=\frac{1}{p}-\frac{1}{q}=\frac{1}{r},$ we obtain 
		\begin{align}
		\left( \sum_{\pi \in \widehat{K}} k_\pi^2  \|\sigma_A(\pi)\|_{\text{op}}^{q'}\left( \frac{ \|\widehat{f}(\pi)\|_{\textnormal{HS}}}{ \sqrt{k_\pi}} \right)^{q'} \right)^{\frac{1}{q'}} \lesssim  \left( \sup_{y>0} y \sum_{\overset{\pi \in \widehat{K}}{\|\sigma(\pi)\|_{op}\geq y}} k_\pi^2 \right)^{\frac{1}{r}} \|f\|_{L^p(K)},\,\, f \in L^p(K).
		\end{align}
		Further, it can be easily checked that 
		\begin{align*}
		\left( \sup_{y>0} y \sum_{\overset{\pi \in \widehat{K}}{\|\sigma(\pi)\|_{\text{op}}\geq y}} k_\pi^2 \right)^{\frac{1}{r}} &= \left( \sup_{y>0} y \sum_{\overset{\pi \in \widehat{K}}{\|\sigma_A(\pi)\|_{\text{op}}^r \geq y}} k_\pi^2 \right)^{\frac{1}{r}} = \left( \sup_{y>0} y^r \sum_{\overset{\pi \in \widehat{K}}{\|\sigma_A(\pi)\|_{\text{op}} \geq y}} k_\pi^2 \right)^{\frac{1}{r}} \\&=  \sup_{y>0} y \left( \sum_{\overset{\pi \in \widehat{K}}{\|\sigma_A(\pi)\|_{op} \geq y}} k_\pi^2 \right)^{\frac{1}{r}}.
		\end{align*}
		Therefore, 
		$$ \|Af\|_{L^q(K)} \lesssim \sup_{y>0} y \left( \sum_{\overset{\pi \in \widehat{K}}{\|\sigma_A(\pi)\|_{\text{op}} \geq y}} k_\pi^2 \right)^{\frac{1}{r}} \|f\|_{L^p(K)}$$ and hence 
		$$\|A\|_{L^p(K) \rightarrow L^q(K)} \lesssim \sup_{y>0} y \left( \sum_{\overset{\pi \in \widehat{K}}{\|\sigma_A(\pi)\|_{\text{op}}>y}} k_\pi^2 \right)^{\frac{1}{p}-\frac{1}{q}},$$
		completing the proof.
	\end{proof}
\begin{rem} Recall that if $\omega(M):= \sum_{\pi \in M} k_\pi^2,$\,\, $M \subseteq \widehat{K},$ is the Plancherel measure on  $\widehat{K}$ then we can interpret the condition \eqref{hocon} in a similar form as in H\"ormander's theorem for $\mathbb{R}^n$ (\cite{Hormander1960}) as follow:
	\begin{equation} \label{alt. Hor}
	\|A\|_{L^p(K) \rightarrow L^q(K)} \leq \sup_{s>0} \left\{s\,\, \omega \{\pi \in \widehat{K}\,:\, \|\sigma_{A}(\pi)\|_{\text{op}}>s \}\right\}^{\frac{1}{p}-\frac{1}{q}}.
	\end{equation}
	We note that condition \eqref{hocon} is sharp for $p=q=2.$ Indeed, first note that, using the Plancherel identity, we have 
	\begin{align} \label{l2}
	\|A\|_{L^2(K) \rightarrow L^2(K)} &= \sup_{\overset{f \in L^2(K)}{\|f\|_2=1}} \|Af\|_{L^2(K)}=\sup_{\overset{f \in L^2(K)}{\|f\|_2=1}} \|\widehat{Af}\|_{\ell^2(\widehat{K})} =\sup_{\overset{f \in L^2(K)}{\|f\|_2=1}} \left(\sum_{\pi \in \widehat{K}} k_\pi \|\sigma_A(\pi) \widehat{f}(\pi)\|_{\textnormal{HS}}^2 \right)^{\frac{1}{2}} \nonumber
\\	& \leq  \sup_{\pi \in \widehat{K}} \|\sigma_{A}(\pi)\|_{\text{op}}  \sup_{\overset{f \in L^2(K)}{\|f\|_2=1}} \left( \sum_{\pi \in \widehat{K}} k_\pi \|\widehat{f}(\pi)\|_{\textnormal{HS}}^2 \right)^{\frac{1}{2}} = \sup_{\pi \in \widehat{K}} \|\sigma_{A}(\pi)\|_{\text{op}}.
	\end{align}
Now, observe that the set $\{\pi \in \widehat{K}: \|\sigma_{A}(\pi)\|_{\text{op}} \geq s \}$ is empty for $s>\|A\|_{L^2(K) \rightarrow L^2(K)}$ in view of \eqref{l2} and, therefore, we have 
\begin{align*}
\|A\|_{L^2(K) \rightarrow L^2(K)} &\leq \sup_{s>0} s \left( \sum_{\overset{\pi \in \widehat{K}}{\|\sigma_A(\pi)\|_{op}>y}} k_\pi^2 \right)^{\frac{1}{2}-\frac{1}{2}} \\& = \sup_{0<s \leq \|A\|_{L^2(K) \rightarrow L^2(K)}} s \cdot\, 1 \leq \|A\|_{L^2(K) \rightarrow L^2(K)}.
\end{align*}
Therefore, we obtained an equality in \eqref{hocon} for $p=q=2.$
\end{rem}

	\begin{cor}
Let $1<p , q <\infty$ and suppose that $A$ is a Fourier multiplier with symbol $\sigma_{A}$ on a compact hypergroup $K.$  If $1<p,q\leq 2,$ then
\begin{align*}
 \|A\|_{L^p(K) \rightarrow L^{q}(K)} \lesssim \sup_{y>0} y\left(  \sum_{\overset{\pi \in \widehat{K}}{\|\sigma_A(\pi)\|_{op}\geq y}} k_\pi^2  \right)^{\frac{1}{p}-\frac{1}{2}},  
\end{align*}
while for $2\leq p,q<\infty$ we have
\begin{align*}
  \|A\|_{L^p(K) \rightarrow L^{q}(K)}\lesssim  \sup_{y>0} y\left(  \sum_{\overset{\pi \in \widehat{K}}{\|\sigma_A(\pi)\|_{op}\geq y}} k_\pi^2  \right)^{\frac{1}{q'}-\frac{1}{2}}.
\end{align*}
\end{cor}
\begin{proof}   Let us assume that   $1<p,q \leq 2.$ Using the compactness of $K,$ we have $\|A\|_{L^p(K) \rightarrow L^{q}(K)}\lesssim \|A\|_{L^p(K) \rightarrow L^{2}(K)}$ and therefore, Theorem \ref{Lp-Lqmulti} gives
\begin{align*}
  \|A\|_{L^p(K) \rightarrow L^{q}(K)}\lesssim \|A\|_{L^p(K) \rightarrow L^{2}(K)} \lesssim  \sup_{y>0} y\left(  \sum_{\overset{\pi \in \widehat{K}}{\|\sigma_A(\pi)\|_{op}\geq y}} k_\pi^2  \right)^{\frac{1}{p}-\frac{1}{2}}.
\end{align*} Now, let us assume that $2\leq p,q<\infty.$ Then $1<p',q'\leq 2,$ and using the first part of the proof we deduce
 \begin{align*}
  \|A\|_{L^p(K) \rightarrow L^{q}(K)}= \|A^*\|_{L^{q'}(K) \rightarrow L^{p'}(K)} \lesssim  \sup_{y>0} y\left(  \sum_{\overset{\pi \in \widehat{K}}{\|\sigma_A(\pi)\|_{op}\geq y}} k_\pi^2  \right)^{\frac{1}{q'}-\frac{1}{2}}.
\end{align*}Thus, we finish the proof.
\end{proof}

	\section{Examples of hypergroups}
	
	In this section we discuss the results obtained in previous sections and prove some new results for two important classes of hypergroups, namely, the  conjugacy classes of the compact non-abelian Lie group $\text{SU(2)}$ and countable compact hypergroups introduced and studied by Dunkl and Ramirez \cite{dun2}.

	\subsection{Conjugacy classes of compact Lie groups}
	Let $G$ be a compact non abelian (Lie) group. Denote the set of all conjugacy classes of $G$ by $\text{Conj}(G),$ that is, $\text{Conj}(G):=\{C_x: x \in G\},$ where for each $x \in G$ the conjugacy class of $x$ is given by  $C_x:=\{yxy^{-1}: y \in G\}.$ The set $\text{Conj}(G)$ equipped with the topology induced by the natural map $q: x \mapsto C_x,$ is a compact Hausdorff space. The compact Hausdorff space $\text{Conj}(G)$ becomes a commutative hypergroup \cite[Section 8]{Jewett} with respect to the convolution defined by, for $x, y \in G,$
	\begin{equation}
	\delta_{C_x}*\delta_{C_y}= \int_G \int_G \delta_{C_{txt^{-1}sys^{-1}}}\, dt\,ds.
	\end{equation}
	Let $\widehat{G}$ be the unitary dual of $G.$ Suppose that each $\pi \in \widehat{G}$ has dimension $d_{\pi}$ and trace $\psi_\pi.$ The fuctions $\psi_\pi$ are called the characters of $G$ but the hypergroup characters are normalized by dividing $\psi_\pi$ by $d_{\pi}.$ More precisely, the hypergroup characters $\chi_\pi$ are obtained by the following relation: $\chi_\pi \circ q= d_{\pi}^{-1} \psi_\pi,$ where $q$ is the natural map $x \mapsto C_x.$ Then the dual $\widehat{\text{Conj}(G)}$ of the commutative hypergroup $\text{Conj}(G)$ is given by: $\widehat{\text{Conj}(G)}:= \{\chi_\pi: \pi \in \widehat{G} \}.$ In fact, the map $\pi \mapsto d_\pi^2 \psi_\pi$ is a bijection between $\widehat{G}$ and $\widehat{\text{Conj}(G)}.$ The Haar measure $\lambda$ of $\widehat{\text{Conj}(G)}$ is induced from the Haar measure of $G$ by the map $q.$ The Haar measure on $\widehat{\text{Conj}(G)}$ is given by 
	$$\omega(\chi_\pi):= k_{\chi_\pi}= d_{\pi}^{2}.$$ 
	
	In the sequel of the paper we will consider the case when $G=\text{SU}(2),$ the compact group of all $2 \times 2$ special unitary matrices. The representation theory of $\text{SU(2)}$ is well established. One can refer to \cite{HR, Vilenkin, RuzT} for more details. Denote the commutative hypergroup $\text{Conj(\text{SU}(2))}$ by $K.$ We identify $K$ with $[0, 1]$ where $t$ in $[0, 1]$ corresponds to the conjugacy class containing the matrix $$ \begin{bmatrix} 
	\exp{(i \pi t)} & 0 \\
	0 &  \exp{(-i \pi t)}
	\end{bmatrix},$$
	see \cite[15.4]{Jewett}. The dual of $\text{SU}(2)$ can be represented by $$\{\pi_l \in \text{Hom}(\text{SU}(2), \text{U}(2l+1)) : l \in \frac{1}{2}\mathbb{N}_0 \},$$ where $\text{U}(d)$ is the unitary matrix group. The number $l \in \frac{1}{2}\mathbb{N}_0$ is called the quantum number. The character $\psi_{l},$ defined as the trace of $\pi_l,$ is computed at $t \in [0, 1]$ given by      $$\psi_l(t)= \frac{\sin(2l+1) \pi t}{\sin \pi t}.$$
	Therefore, the dual $\widehat{K}$ is given by $\{(2l+1)^{-1} \psi_l: l \in \frac{1}{2} \mathbb{N}_0\}$ and $k_{\chi_l}= (2l+1)^2.$
	
	The Paley inequality in Theorem \ref{Paley} takes the following form in the setting of the compact abelian hypergroup $\text{Conj(SU(2))}.$
	\begin{thm}
		Let $1 < p \leq 2$ and let $\{\varphi(l)\}_{l \in \frac{1}{2}\mathbb{N}_0}$ be a positive sequence  such that 
		$$M_\varphi:= \sup_{y>0} y \sum_{\overset{l \in \frac{1}{2}\mathbb{N}_0}{\varphi(l) \geq y}} (2l+1)^4<\infty.$$
		Then we have 
		$$\sum_{l \in \frac{1}{2}\mathbb{N}_0} (2l+1)^{4-p} \widehat{f}(l) \varphi(l)^{2-p} \lesssim M_\varphi^{2-p} \|f\|_{L^p(\text{Conj}(SU(2)))}^p.$$
	\end{thm}
	
	We have the following Hardy-Littlewood inequality for the commutative hypergroup $\text{Conj(SU)(2)}.$
	\begin{thm} \label{HLineq}
		If $1 < p \leq 2$ and $f \in L^p(\textnormal{Conj(SU)(2)}),$ then there exists a universal constant $C=C(p)$ such that 
		\begin{equation} \label{37}
		\sum_{l \in \frac{1}{2}\mathbb{N}_0} (2l+1)^{5p-8} |\widehat{f}(l)|^{p} \leq C \|f\|_{L^p(\textnormal{Conj(SU)(2)})}.
		\end{equation}
	\end{thm}
	\begin{proof} 
		Take  $\beta=3=\dim(\text{SU(2)})$ and $\{\mu_{\chi_\pi}\}_{\pi \in \widehat{\textnormal{Conj(SU)(2)}}}:=\{(2l+1)^2\}_{l \in \frac{1}{2} \mathbb{N}_0}.$ Then the condition \eqref{HLabecon} turns out to be 
		$$\sum_{l \in \frac{1}{2}\mathbb{N}_0} \frac{(2l+1)^4}{(|(2l+1)^2|)^3}= \sum_{l \in \frac{1}{2}\mathbb{N}_0} \frac{1}{(2l+1)^2}=\frac{\pi^2}{6}$$ which is finite. Therefore, by Theorem \ref{HLabe} the proof of inequality \eqref{37} follows. 
	\end{proof}
	
	\begin{rem}
	    We would like to recall here the Hardy-Littlewood inequality on the compact Lie group $\text{SU}(2)$ obtained by the second author and R. Akylzhanov in \cite{ARN1}, which says that for $1< p \leq 2$ and $f \in L^p(\text{SU}(2))$ we have 
	   $$\sum_{l \in \frac{1}{2}\mathbb{N}_0} (2l+1)^{\frac{5}{2}p-4} \|\widehat{f}(l)\|_{\textnormal{HS}} \leq C_p \|f\|_{L^p(\text{SU}(2))}.$$
	   In view of this inequality the Hardy-Littlewood inequality for the compact commutative hypergroup $\text{Conj(SU(2))}$ above is a suitable analogue because in $\text{Conj(SU(2))}$ the dimension $(2l+1)$ of the representation $\pi_l$  is replaced by hyperdimension $(2l+1)^2$ of $\pi_l$ and Fourier transform $f$ at $l \in \frac{1}{2}\mathbb{N}_0$ is scalar so $\|\widehat{f}(l)\|_{\textnormal{HS}}$ is just $|\widehat{f}(l)|.$
	   \end{rem}
Using the duality, we get the following corollary.
	
	\begin{cor} \label{5.3cor}
		If  $2 \leq p <\infty$ and $\sum_{l \in \frac{1}{2}\mathbb{N}_0} (2l+1)^{5p-8} |\widehat{f}(l)|^{p}<\infty$ then $$f \in L^p(\textnormal{Conj(SU)(2)}).$$ Moreover, we have 
		$$\|f\|_{L^p(\textnormal{Conj(SU)(2)})} \leq C(p) \sum_{l \in \frac{1}{2}\mathbb{N}_0} (2l+1)^{5p-8} |\widehat{f}(l)|^{p}.$$
	\end{cor}
	\begin{proof} Using the duality of $L^p$-spaces, we have 
		$$\|f\|_{L^p(\textnormal{Conj(SU)(2)})} = \sup_{\underset{\|g\|_{L^{p'}(\textnormal{Conj(SU)(2)})} \leq 1}{g \in L^{p'}(\textnormal{Conj(SU)(2)})}} \left|\int_{\textnormal{Conj(SU)(2)}} f(x)\, \overline{g(x)}\, d\lambda(x)\right|. $$
		Now, by the Plancherel identity \eqref{pabel}, we get
		$$ \int_{\textnormal{Conj(SU)(2)}} f(x) \overline{g(x)}\, d\lambda(x) = \sum_{l \in \frac{1}{2}\mathbb{N}_0} (2l+1)^2 \widehat{f}(l)\,\overline{\widehat{g}(l)}.$$
		By noting that $(2l+1)^2=(2l+1)^{2\left(\frac{5}{2}-\frac{4}{p}+\frac{5}{2}-\frac{4}{p'} \right)}$ and applying the H\"older inequality, for any $g \in L^{p'}(\textnormal{Conj(SU)(2)}),$  we have 
		\begin{align*}
		\left| \sum_{l \in \frac{1}{2} \mathbb{N}_0} (2l+1)^2 \widehat{f}(l)\, \overline{\widehat{g}(l)} \right| &\leq \sum_{l \in \frac{1}{2} \mathbb{N}_0} (2l+1)^{5-\frac{8}{p}} |\widehat{f}(l)| (2l+1)^{5-\frac{8}{p}} |\widehat{g}(l)| \\& \leq  \left( \sum_{l \in \frac{1}{2} \mathbb{N}_0} (2l+1)^{5p-8} |\widehat{f}(l)|^p \right)^{\frac{1}{p}} \left( \sum_{l \in \frac{1}{2} \mathbb{N}_0} (2l+1)^{5p'-8} |\widehat{g}(l)|^{p'} \right)^{\frac{1}{p'}} \\&\leq C(p) \left( \sum_{l \in \frac{1}{2} \mathbb{N}_0} (2l+1)^{5p-8} |\widehat{f}(l)|^p \right)^{\frac{1}{p}} \,\, \|g\|_{L^{p'}(\textnormal{Conj(SU)(2)}},
		\end{align*} where we have used Theorem \ref{HLineq} in the last inequality. 
		Therefore, by \eqref{pabel} we have
		
		$$  \left|\int_{\textnormal{Conj(SU)(2)}} f(x)\, \overline{g(x)}\, d\lambda(x)\right| \leq C(p) \left( \sum_{l \in \frac{1}{2} \mathbb{N}_0} (2l+1)^{5p-8} |\widehat{f}(l)|^p \right)^{\frac{1}{p}} \,\, \|g\|_{L^{p'}(\textnormal{Conj(SU)(2)}}.$$
		Thus, by taking supremum over all $g \in L^{p'}(\textnormal{Conj(SU)(2)})$ with $\|g\|_{L^{p'}(\textnormal{Conj(SU)(2)})} \leq 1,$ we get
		$$\|f\|_{L^p(\textnormal{Conj(SU)(2)})} \leq C(p) \left( \sum_{l \in \frac{1}{2} \mathbb{N}_0} (2l+1)^{5p-8} |\widehat{f}(l)|^p \right)^{\frac{1}{p}},$$ completing the proof. 
	\end{proof}

	\subsection{Countable compact hypergroups} Dunkl and Ramirez \cite{dun2} studied an interesting class of countable hypergroups.
	Let $\mathbb{N}_0^*=\{0,1,2, \ldots, \infty\}$ be the one-point compactification of $\mathbb{N}_0.$ Dunkl and Ramirez \cite{dun2} defined a convolution structure $*$ on $\mathbb{N}_0^*$ for every $ 0<a \leq \frac{1}{2},$ denoted by $H_a,$ to make it a (hermitian) countable compact hypergroup . For a prime $p,$ let $\Delta_p$ be the ring of p-adic integers and $\mathcal{W}$ be its group of units, that is, $\{x=x_0+x_1p+ \ldots+ x_np^n+ \ldots \in \Delta_p : x_j = 0,1, \ldots,p-1 \, \text{for}\, j \geq 0 \, \text{and} \, x_0 \neq 0  \}$. For $a=\frac{1}{p},$ $H_{\frac{1}{p}}$ derives its structure from $\mathcal{W}$-orbits of the action of $\mathcal{W}$ on $\Delta_p$ by multiplication in $\Delta_p.$  In fact, the convolution is given as follows: for $m, n \in \mathbb{N}_0,$ define
	$$\delta_m*\delta_n= \delta_{\text{min}\{ m,n\}}\,\,\,\,\,\, \text{if}\, m \neq n,$$ $\delta_m*\delta_\infty= \delta_\infty*\delta_m=\delta_m$, $\delta_\infty*\delta_\infty=\delta_\infty,$ 
	and for $m=n,$ 
	$$\delta_m* \delta_m(t)= \begin{cases} 0 & t<m, \\ \frac{1-2a}{1-a} & t=m, \\ a^k & t=m+k>m, \\ 0 & t= \infty. \end{cases}$$
	The Haar measure $\lambda$ on $H_a$ is given by 
	$$\lambda(\{k\})=a^k(1-a)\quad \text{for}\,\, k<\infty, \quad \lambda(\{\infty\})=0.$$ The elements of $\widehat{H_a}$ are given by $\{\chi_n : n \in \mathbb{N}_0\},$ where, for $k \in H_a,$ \begin{eqnarray*}
		\chi_n(k)= \begin{cases}  0 & \text{if}\,\,k <n-1, \\ \frac{a}{a-1} & \text{if}\,\,k=n-1, \\ 1 &\text{if}\,\, k \geq n \,\,\, (\text{or}\,\, k = \infty).\end{cases}
	\end{eqnarray*}
	Then the convolution `$*$' on $\mathbb{N}_0$ identified with $\widehat{H_a}=\{\chi_n : n \in \mathbb{N}_0\}$ is dictated by pointwise product of functions in $\widehat{H_a},$ that is: 
	\begin{eqnarray*}
		\delta_{\chi_m}*\delta_{\chi_n} &=& \delta_{\chi_{\text{max}\{m,n\}}} \,\,\, \text{for}\,\,\, m \neq n, \\
		\delta_{\chi_0}*\delta_{\chi_0}&=&\delta_{\chi_0}, \,\,\,\, \delta_{\chi_1}*\delta_{\chi_1} = \frac{a}{1-a} \delta_{\chi_0}+ \frac{1-2a}{1-a} \delta_{\chi_1},\\
		\delta_{\chi_n}*\delta_{\chi_n}&=& \frac{a^n}{1-a} \delta_{\chi_0}+ \sum_{k=1}^{n-1} a^{n-k} \delta_{\chi_k}+\frac{1-2a}{1-a} \delta_{\chi_n} \,\,\,\,\, \text{for}\, n \geq 2.
	\end{eqnarray*} The dual space $\widehat{H_a}$ of $H_a$ turns into a hermitian discrete hypergroup with respect to the above convolution. The Plancherel measure $\omega$ on $\widehat{H}_a$ is given by 
	$$\omega(\chi_0)=1\quad \text{and} \quad \omega(\chi_n)=(1-a)a^{-n}\quad\text{for}\,\, n \geq 1.$$
	
	The Paley-type inequality for Dunkl-Ramirez hypergroup is then given by the following theorem. 
	\begin{thm}
		Let $1 < p \leq 2$ and let $\{\varphi(n)\}_{n \in \mathbb{N}_0}$ be a positive sequence  such that 
		$$M_\varphi:= \sup_{y>0} y \sum_{\overset{n \in \mathbb{N}}{\varphi(n) \geq y}} (1-a)^2a^{-2n}+\varphi(0)<\infty.$$
		Then we have 
		$$\sum_{n \in \mathbb{N}} (a^{-n}(1-a))^{2-\frac{p}{2}} \widehat{f}(n) \varphi(n)^{2-p} \lesssim M_\varphi^{2-p} \|f\|_{L^p(H_a)}^p.$$
	\end{thm}
	
	We have the following Hardy-Littlewood inequality for the compact countable commutative hypergroups $H_a.$
	\begin{thm} \label{HLineqduk}
		If $1 < p \leq 2$  then there exists a constant $C=C(p)$ such that 
		\begin{equation} \label{39}
		f(0)+\sum_{n \in \mathbb{N}} ((1-a)a^{-n})^{p(\frac{5}{2}-\frac{4}{p})} |\widehat{f}(n)|^{p} \leq C \|f\|_{L^p(H_a)}.
		\end{equation}
	\end{thm}
\begin{proof} We apply Theorem \ref{HLabe} to get  inequality \eqref{39} above. The condition \eqref{HLabecon} for $\beta=3$ by choosing the sequence $\{\mu_{\chi_n}\}_{n \in \mathbb{N}}:= \{(1-a)a^{-n}\}_{n \in \mathbb{N}}$ with $\mu_{\chi_0}=1$ turns out to be 
	$$\sum_{n \in \mathbb{N}_0} \frac{k_{\chi_n}^2 }{|\mu_{\chi_n}|^\beta}=\sum_{n \in \mathbb{N}_0} \frac{(1-a)^2 a^{-2n}}{(1-a)^3 a^{-3n}}=1+ \frac{1}{1-a} \sum_{n \in \mathbb{N}} a^n= \frac{1-a+a^2}{(1-a)^2}=\frac{(1-a)^2-a}{(1-a)^2},$$
	which is finite. Therefore, by Theorem \ref{HLabe} the proof of inequality \eqref{39} follows.
	\end{proof}
The proof of the following corollary is exactly similar to Corollary \ref{5.3cor} in the previous subsection. 
	\begin{cor}
		If  $2 \leq p <\infty$ and $f(0)+\sum_{n \in \mathbb{N}} ((1-a)a^{-n})^{p(\frac{5}{2}-\frac{4}{p})} |\widehat{f}(n)|^{p}<\infty,$ then $$f \in L^p(H_a).$$ Moreover, we have 
		$$\|f \|_{ L^p(H_a)} \leq C_p \left(f(0)+\sum_{n \in \mathbb{N}} ((1-a)a^{-n})^{p(\frac{5}{2}-\frac{4}{p})} |\widehat{f}(n)|^{p} \right).$$
	\end{cor}

\section*{Acknowledgement}
Vishvesh Kumar thanks Prof. Ajit Iqbal Singh and Prof. Kenneth A. Ross for their suggestions and comments.  The authors are supported by FWO Odysseus 1 grant G.0H94.18N: Analysis and Partial Differential Equations. Michael Ruzhansky is also supported in parts by the EPSRC Grant EP/R003025/1 and by the Leverhulme Research Grant RPG-2017-151.

\end{document}